\documentclass[11pt]{article}

\usepackage{geometry}   
\usepackage{amsmath}
\usepackage{amssymb}
\usepackage{amsthm}

\usepackage{graphicx} 
\usepackage{tikz}
\usetikzlibrary{matrix,arrows,positioning}

\newcounter{remark}
\newcommand{\remark}{\addtocounter{remark}{1}
                       \par \quad {\bf \arabic{remark}}.\,
                      }
\newenvironment{rks}{\begin{quote}
                     \normalfont\footnotesize
                     \setcounter{remark}{0}
                     \setlength{\parskip}{0.25\parskip}
                     \renewcommand{\item}{\remark}
                     {\bf Remarks}
                     \nobreak
                    }
                    {\end{quote}}

\newenvironment{rk}{\begin{quote}
                     \normalfont\footnotesize {{\bf Remark} --}
                    }{\end{quote}}

\newcommand{\RR}{\mathbb R}

\newcommand{\FF}{\mathbb F}

\newcommand{\Ecal}{\mathcal E}

\newtheorem{theo}{Theorem}[section]

\newtheorem{defi}[theo]{Definition}

\newtheorem{lem}[theo]{Lemma}

\newtheorem{prop}[theo]{Proposition}

\newcommand{\Gram}{\operatorname{Gram}}

\newcommand{\Id}{\operatorname{Id}}

\newcommand{\Num}{\operatorname{Num}}

\newcommand{\Vect}{\operatorname{Vect}}

\newcommand{\ideng}[1]{\left\langle #1 \right\rangle}

\newcommand{\HVperp}{{\mathcal E}}

\newcommand{\gammaa}{\gamma_{12}}

\newcommand{\scalaire}[2]{\left\langle #1, #2 \right\rangle}
\newcommand{\phietoilebas}[2]{{\varphi_{#1/#2}}_*}

\title{Number of points of curves over finite fields in some relative situations from an euclidean point of vue}
\author{
Emmanuel Hallouin \& Marc Perret\thanks{Institut de Math\'ematiques de Toulouse~; UMR 5219, Universit\'e de Toulouse~; CNRS, UT2J, F-31058 Toulouse, France, hallouin@univ-tlse2.fr, perret@univ-tlse2.fr.
Funded by ANR grant  ANR-15-CE39-0013-01 ``manta"}
}

\date{\today}

\begin{document}

\maketitle

\begin{abstract} 
We study the number of rational points of smooth projective curves over finite fields in some relative situations in the spirit of a previous paper~\cite{WeilNous} from an euclidean point of vue. We prove some kinds of \emph{relative Weil bounds}, derived from Schwarz inequality for some ``relative parts" of the diagonal and of the graph of the Frobenius on some euclidean sub-spaces of the numerical space of the squared curve endowed with the opposite of the intersection product.
\end{abstract}




\section*{Introduction}
Several general bounds on the number $\sharp X(\FF_q)$ of rational points on absolutely irreducible smooth projective curves $X$ of genus $g_X$ defined over the finite field $\FF_q$ are known, the most famous being Weil bound~\cite{WeilCA} that
\begin{equation} \label{Weil}
\vert \sharp X(\FF_q)-(q+1)\vert \leq 2g_X\sqrt{q}.
\end{equation}
Other bounds are known, such as asymptotic Drinfel'd-Vladut one~\cite{VD} and Tsfasman one~\cite{T92}, or a relative bound (for instance in~\cite{AubryPerret})
\begin{equation} \label{Weil-relatif}
\vert \sharp X(\FF_q)-\sharp Y(\FF_q)\vert \leq 2(g_X-g_Y)\sqrt{q}
\end{equation}
in case there exists a covering $X \longrightarrow Y$. Twisting a little bit Weil's original proof~\cite{WeilCA} of~(\ref{Weil}), the authors have given in a previous paper~\cite{WeilNous} proofs of Weil's, Drinfeld-Vladut's, Tsfasman's and some other new bounds from an euclidean point of vue. For instance, Weil bound~(\ref{Weil}) is only Schwarz inequality for two very natural vectors, namely~$\gamma^0_X$ coming from the class of the diagonal~$\Delta_X$ inside $X\times X$, and~$\gamma^1_X$ coming from the class of the graph~$\Gamma_{F_X}$ of the Frobenius morphism~$F_X$ on~$X$, lying in some euclidean subspace
$$\Ecal_X=\Vect(H_X, V_X)^{\perp} \subset \Num(X\times X)_{\RR}$$
for the opposite of the intersection product on the real numerical vector space~$\Num(X\times X)_{\RR}$ of the cartesian surface~$X\times X$, where $H_X$ and $V_X$ are respectively the horizontal and vertical classes. The aim of this paper is to complete this work in some relative situations, giving for instance a similar euclidean proof for~(\ref{Weil-relatif}).

\bigskip

A key point is that a covering~$f : X \longrightarrow Y$ induces a pull-back linear morphism~$(f\times f)^*$ and a push-forward linear morphism~$(f\times f)_*$ between~$\Ecal_X$ and~$\Ecal_Y$. Both morphisms behave in some very pleasant way with respect to the vectors~$\gamma_X^i$ and~$\gamma_Y^i$ for $i=0, 1$, in such a way that it can be said that~$\gamma_X^i$ is the orthogonal sum, in~$\Ecal_X$, of the pull-back of~$\gamma_Y^i$ and of some ``relative part"~$\gamma_{Y/X}^{i}$. The Gram matrix between~$\gamma_{Y/X}^{0}$ and~$\gamma_{Y/X}^{1}$ can be computed, and~(\ref{Weil-relatif}) is only Schwarz inequality for this pair of vectors.

This point of vue can be pushed further in a commutative diagram~(\ref{DiagrXY1Y2Z}) below. A relative part~$\gamma_{X/Y_1,Y_2/Z}^i$  of~$\gamma_X^i$ we denote by~$\gammaa^{i}$ for simplicity can be defined, and Schwarz for $i=0, 1$ gives Theorem~\ref{final}, a new bound relating the number of rational points of the four curves involved in case the fibre product is absolutely irreducible and smooth.

Notice that if~(\ref{Weil-relatif}) can be proved using Tate modules of the jacobians of the involved curves (see e.g.~\cite{AubryPerret}), Proposition~\ref{RelatifNonTrivial} and Theorem~\ref{final}, up to our knowledge, cannot.

\section{Known absolute results \cite{WeilNous}}

In this first section, we gather the notations and results of our previous work \cite{WeilNous} that are needed in this paper.

\medskip

Let $X$ be an absolutely irreducible smooth projective curve of genus~$g$ defined over the finite field~$\FF_q$ with~$q$ elements. Weil's proof of Rieman hypothesis in this context rests on intersection theory on the numerical space~$\Num (X\times X)_{\RR}$ of the algebraic surface~$X\times X$. The key point is the Hodge Index Theorem stating that the intersection pairing is definite negative on the orthogonal complement of the class of an ample
divisor \cite[Chap~V,Th.1.9 \& Rk.1.9.1]{Hartshorne}. In particular the opposite of the intersection pairing defines a scalar product
on the orthogonal complement of the plane generated by the classes of the horizontal and the vertical divisors since their sum is ample. This motivates the following definition.

\begin{defi}
  Let~$H_X$ and~$V_X$ be the horizontal and vertical classes inside~$\Num (X\times X)_{\RR}$. We put:
$$
\Ecal_X = \Vect\left(H_X,V_X\right)^\perp
$$
and we endow this vector space with the scalar product defined by~$\scalaire{D_1}{D_2} = -D_1 \cdot D_2$, the opposite of the intersection
pairing~$D_1 \cdot D_2$  on~$X\times X$.
\end{defi}

It is usefull to introduce the orthogonal projection of~$\Num (X\times X)_{\RR}$ onto~$\Ecal_X$ for the intersection pairing bilinear form:
\begin{equation} \label{proj}
\begin{matrix}
p_X: & \Num (X\times X)_{\RR} & \longrightarrow &\Ecal_X\\
~	& D	& \longmapsto &D-(D\cdot V_X) H_X-(D\cdot H_X) V_X.
\end{matrix}
\end{equation}
In this context, the family of (orthogonal projections of) graphs of iterates of the $q$-Frobenius morphism play a crucial role.

\begin{defi} \label{def_gamma_s}
Let~$F_X : X \to X$ be the $q$-Frobenius morphism on the curve~$X$. For~$i\geq 0$, let~$\Gamma_{F_X}^i$
be the class in~$\Num(X\times X)_\RR$ of the graph of the $i$-th iterate of~$F_X$ (the $0$-th iterate beeing identity).
By projecting, we put:
$$
\gamma_X^i = p_X\left(\Gamma_{F_X}^i\right)
\in \HVperp_X,
$$
where~$p_X : \Num(X\times X) \to \Ecal_X$ is the orthogonal projection onto~$\Ecal_X$ given by~(\ref{proj}).
\end{defi}

\begin{rk}
We delete here the normalization of the vectors $\gamma^i_X$ introduced in our previous work~\cite[Definition 4]{WeilNous}, necessary therein for some intersection matrix to be Toeplitz~\cite[Proposition~5]{WeilNous}. This particular shape of the intersection matrix is irrelevant in the present work.
\end{rk}
\bigskip


The computation of the norms and the scalar products of the~$\gamma_X^i$'s is well known and can be found in our previous
work~\cite[Proposition 5]{WeilNous} in which another normalization is used.

\begin{lem}\label{lem_norm_paring_gammaX}
The norms and the scalar products of the~$\gamma_X^i$'s are given by
\begin{align} \label{norm_gammaX}  
&\left\|\gamma_X^i\right\|_X = \sqrt{2g_X q^i}
&
&\text{and}
&
&\scalaire{\gamma_{X}^i}{\gamma_{X}^{i+j}}_{X} = 
q^i\left((q^j + 1) - \sharp X(\FF_{q^j})\right)
\end{align}
for any~$i\geq 0$ and~$j\geq 1$.
\end{lem}

\section{The relative case}

We concentrate in this Section on the simplest relative situation. The data
is a finite morphism~$f : X \to Y$ of degree~$d$, where~$X$ and~$Y$ are
 absolutely irreducible smooth projective curves defined over~$\FF_q$,
whose genus are denoted by~$g_X$ and~$g_Y$.

\subsection{The pull-back and push-forward morphisms}

The morphism~$f\times f$ from~$X\times X$ to~$Y\times Y$ induces a push forward morphism
$$
(f\times f)_* : \Num(X\times X)_\RR \longrightarrow \Num(Y\times Y)_\RR
$$
and a pull back morphism
$$(f\times f)^* : \Num(Y\times Y)_\RR \longrightarrow \Num(X\times X)_\RR.
$$
For normalization purpose, it is convenient to define~$\varphi^*_{X/Y}$ and~$\varphi_{*,X/Y}$
 (or~$\varphi^*$ and~$\varphi_*$ for short) by
\begin{equation} \label{phipsi}
\varphi_*=\phietoilebas{X}{Y} = \frac{1}{d} (f\times f)_*
\qquad\text{and}\qquad
\varphi^*=\varphi^*_{X/Y} = \frac{1}{d} (f\times f)^*.
\end{equation}
In the next proposition, it is shown that~$\varphi^*$ sends the euclidean space~$\Ecal_Y$ to~$\Ecal_X$ and that~$\varphi_*$ sends the euclidean
space~$\Ecal_X$ to~$\Ecal_Y$ with some special features. In the sequel we denote the same way the maps~$\varphi^*$ and~$\varphi_*$ and their restrictions to either $\Ecal_X$ or $\Ecal_Y$.

\begin{prop} \label{prop_phipsi}
The morphisms~$\varphi_*$ and~$\varphi^*$ satisfy the following.
\begin{enumerate}
\item\label{item_HVpreserves} Vertical and horizontal divisors are preserved:
\begin{align} \label{HVpreserves}
&\varphi^*(H_Y)=H_X,
&
&\varphi_*(H_X)=H_Y,
&
&\varphi^*(V_Y)=V_X,
&
&\varphi_*(V_X)=V_Y,
\end{align}
so as the orthogonal complements of the horizontal and vertical parts:
\begin{align}\label{EcalStables}
&\varphi^*(\Ecal_Y)\subset \Ecal_X,
&
&\varphi_*(\Ecal_X)\subset \Ecal_Y.
\end{align}
\end{enumerate}
Moreover, the restrictions of~$\varphi^*$ to~$\Ecal_Y$ and of~$\varphi_*$ to~$\Ecal_X$ satisfy:
\begin{enumerate}\setcounter{enumi}{1}
\item \label{ProjForm} {\bfseries [projection formula]} for all
$\gamma \in \Ecal_X$ and all $\delta \in \Ecal_Y$,
$\scalaire{\gamma}{\varphi^*\left(\delta\right)}_{X}= \scalaire{\varphi_*(\gamma)}{\delta}_{Y}$;
\item \label{circ=Id} $\varphi_* \circ \varphi^* = \Id_{\Ecal_Y}$, the identity map on $\Ecal_Y$;
\item \label{embeding} {\bfseries [isometric embeding]} the morphism~$\varphi^*$ is an isometric
embedding of~$\Ecal_Y$ into~$\Ecal_X$;
%
\item \label{Orth_proj} {\bfseries [orthogonal projection]} the map~$\varphi^* \circ \varphi_*$ (restricted to~$\Ecal_X$) is the orthogonal projection
of $\Ecal_X$ onto the subspace~$\varphi^*(\Ecal_Y)$.
\end{enumerate}
\end{prop}

\begin{proof}
For items~\ref{item_HVpreserves} and~\ref{circ=Id},
we first consider the maps~$\varphi^*$ and~$\varphi_*$ with their domain and co-domain equal to the total spaces~$\Num(X\times X)_\RR$
and~$\Num(Y\times Y)_\RR$. Since the morphism~$f : X \to Y$ is finite, it is proper \cite[Chap~II, Ex~4.1]{Hartshorne}; since~$Y$ is a smooth
curve, the morphism~$f$ is also flat \cite[Chap~III, Prop~9.7]{Hartshorne}. Then so is the square
morphism~$f\times f$ \cite[\S1.10, Prop~1.10]{Fulton}. Formulas~\eqref{HVpreserves} follow, so as
that~$(f\times f)_* \circ (f\times f)^* = d^2 \Id_{\Num(Y\times Y)_\RR}$ \cite[\S1.7, Ex~1.7.4]{Fulton}, proving item~\ref{circ=Id} in the way. Moreover the projection
formula \cite[Appen~A, A4]{Hartshorne} asserts that:
$$
\forall D \in \Num(X\times X)_\RR,
\forall C \in \Num(Y\times Y)_\RR,
\qquad
(f\times f)^*(D) \cdot C
=
D \cdot (f\times f)_*(C)
$$
where the first (resp.~second) intersection product is intersection in the surface~$X\times X$ (resp.~$Y\times Y$).
Going back to~$\varphi$, these prove that~$\varphi_* \circ \varphi^* = \Id_{\Num(Y\times Y)_\RR}$ and
that~$\varphi^*(D) \cdot C = D \cdot \varphi_*(C)$. Using formulas~\eqref{HVpreserves}, we
deduce that~$H_X \cdot \varphi^*(D) = H_Y \cdot D$ (the same with~$V_X$, $V_Y$) and thus~$\varphi^*(\Ecal_Y)\subset\Ecal_X$. In the
same way~$\varphi^*(\Ecal_X)\subset\Ecal_Y$, so that item~\ref{item_HVpreserves} is proved.

From now on, we restrict the maps~$\varphi^*$ and~$\varphi_x$ to the subspaces~$\Ecal_Y$ and~$\Ecal_X$ without changing the notations.
item~\ref{ProjForm} is only a restatement of the projection formula above.
Item~\ref{embeding} is an easy consequence of items~\ref{ProjForm} and~\ref{circ=Id}. Last, the
morphism~$\varphi^*\circ\varphi_*$ is by item~\ref{circ=Id} a projector whose image is the space~$\varphi^*(\Ecal_Y)$. For~$\gamma\in\Ecal_X$,
by~items~\ref{ProjForm} and~\ref{circ=Id}, one has
$$
\ideng{\varphi^*\circ\varphi_*(\gamma),\gamma-\varphi^*\circ\varphi_*(\gamma)}_X
=
\ideng{\varphi_*(\gamma),\varphi_*(\gamma)}_Y
-
\ideng{\varphi_*(\gamma),\varphi_*\circ \varphi^*\circ\varphi_*(\gamma)}_Y
=
0
$$
and thus, writing~$\gamma = \varphi^*\circ\varphi_*(\gamma) + \left(\gamma-\varphi^*\circ\varphi_*(\gamma)\right)$, we see that this
is the sum of two orthogonal elements, the first one lying in~$\varphi^*(\Ecal_Y)$
and the second one in~$\varphi^*(\Ecal_Y)^\perp$. This proves item~\ref{Orth_proj}.
\end{proof}

\begin{rks}
\item Deleting the normalization factor $\frac{1}{d}$ in~(\ref{phipsi}), the map $\varphi^*$ would be $(f\times f)^*$, a similitude of modulus $d$ instead of an isometry as in item~\ref{embeding}.
\item Since the pull-back map~$\varphi_{X/Y}^*$ is an isometry (and thus is injective), we could have identified the space~$\Ecal_Y$ with
its embedding~$\varphi_{X/Y}^*(\Ecal_Y)$ inside~$\Ecal_X$.
With this point of view, the push-forward map~$\phietoilebas{X}{Y}$ is truly the orthogonal projection of~$\Ecal_X$ onto~$\Ecal_Y$.
In every proofs in the sequel, the reader may feels more comfortable by skipping all the $\varphi_{\_/\_}^*$ maps and thinking to
the $\phietoilebas{\_}{\_}$ maps as orthogonal projections.
\end{rks}

The ``bottom'' space~$\Ecal_Y$ embeds into the ``top'' space~$\Ecal_X$ via the pull-back morphism~$\varphi_{X/Y}^*$, and the orthogonal complement
of this embedding~$\varphi_{X/Y}^*(\Ecal_Y)$ into~$\Ecal_X$ plays a crucial role in the whole paper.

\begin{defi}\label{defi_E_X/Y}
The orthogonal complement $\varphi_{X/Y}^*(\Ecal_Y)^{\perp}$ of $\varphi_{X/Y}^*(\Ecal_Y)$ inside $\Ecal_X$ is denoted by~$\Ecal_{X/Y}$ and is called the relative
space for the covering $X \rightarrow Y$.

\end{defi}

We emphasize for future need the fact that this space~$\Ecal_{X/Y}$ is contained in the kernel of the push-forward morphism.

\begin{lem} \label{phi_*zero}
The push-forward morphism~$\phietoilebas{X}{Y}$ is zero on the relative space~$\Ecal_{X/Y}$ for $X \rightarrow Y$.
\end{lem}

\begin{proof}
Let $\gamma \in \Ecal_{X/Y}= \varphi^*(\Ecal_Y)^{\perp}$. Then, $\varphi^*\circ \varphi_*(\gamma) =0$ by Proposition~\ref{prop_phipsi} item~\ref{Orth_proj}, so that $\varphi^*(\gamma) =0$ by item~\ref{embeding}. 
\end{proof}

\subsection{The relative part of the~$\gamma_X^i$'s in a covering} \label{s_relative_XtoY}

In this section, we look at the image of the iterated Frobenius graphs and their orthogonal projection into the spaces~$\Ecal_{-}$
(see Definition~\ref{def_gamma_s}) under the maps~$\varphi^*$ and~$\varphi_*$.

First, for any~$i \geq 0$, one has
\begin{equation}\label{phigamma}
\phietoilebas{X}{Y}(\gamma_X^i) = \gamma_Y^i,
\end{equation}
a consequence of equality~$(f\times f)_* (\Gamma_{F^i_X}) = d\Gamma_{F^i_Y}$ and of formula~\eqref{proj}
for the projection~$p_X$.

On the other hand, we do not have equality~$\varphi_{X/Y}^*(\gamma_Y^i) = \gamma_X^i$, but rather some orthogonal decomposition as follows.
In view of definition~\ref{defi_E_X/Y}, we have the orthogonal sum
\begin{equation} \label{Decomp_E_X} 
\Ecal_X= \varphi^*(\Ecal_Y) \oplus  \Ecal_{X/Y}.
\end{equation}
For $i \geq 0$, the corresponding decomposition of $\gamma_X^i$ is
\begin{equation} \label{Decomp_gamma}
\gamma^i_X
=
\underbrace{\varphi^*(\gamma_Y^i)}_{\in \varphi^*(\Ecal_Y)}
+
\underbrace{\left(\gamma^i_X-\varphi^*(\gamma^i_Y)\right)}_{\in \varphi^*(\Ecal_Y)^{\perp}},
\end{equation}
since by~Proposition~\ref{prop_phipsi}, item~\ref{ProjForm} together with Formula~(\ref{phigamma}), the orthogonal projection of $\gamma_X^i$ is $\varphi^*(\gamma_Y^i)$.
The orthogonal components $\gamma^i_X-\varphi^*(\gamma^i_Y)$ inside $\Ecal_{X/Y}= \varphi^*(\Ecal_Y)^{\perp}$ turning to be of greatest importance in the sequel, we give them a name in the following Definition.

\begin{defi} \label{Def_gamma_relative}
  For~$i\geq 0$, the component
\begin{equation*} 
\gamma_{X/Y}^{i}= \gamma^i_X-\varphi^*(\gamma^i_Y) \in \Ecal_{X/Y},
\end{equation*}
of $\gamma_X^i$ inside $\Ecal_{X/Y}$ is called the $i$-th relative part of the Frobenius.
\end{defi}


We can relate in the following Lemma the scalar products between the relative parts of $\gamma_X^i$ and $\gamma_X^{i+j}$, for any~$i,j\geq 0$,
to the standard geometrical and arithmetical invariants of both curves~$X$ and~$Y$.

\begin{lem} \label{normeetpsa2}
For any $i \geq 0$ and $j>0$, we have
\begin{align*}
&\left\|\gamma_{X/Y}^{i}\right\|_X
=
\sqrt{2(g_{X}-g_{Y})q^i}
&
&\text{and}
&
&\scalaire{\gamma_{X/Y}^{i}}{\gamma_{X/Y}^{i+j}}_{X}
= q^i \left(\sharp Y(\FF_{q^j})-\sharp X(\FF_{q^j})\right).
\end{align*}
\end{lem}

\begin{proof}
Since~$\gamma_{X/Y}^{i} \perp \varphi^*(\gamma_Y^i)$, the first norm calculation is just Pythagore Theorem.
Indeed, we have for any $i\geq 0$
\begin{align*}
\left\|\gamma_{X}^i\right\|_X^2
&=
\left\|\varphi^*(\gamma_{Y}^i)\right\|_X^2
+
\left\| \gamma_{X/Y}^{i} \right\|_X^2
&& \text{by Def.~\ref{Def_gamma_relative} and Pythagore}\\
&=
\left\|\gamma_{Y}^i\right\|_Y^2
+
\left\| \gamma_{X/Y}^{i} \right\|_X^2,
&& \text{since $\varphi^*$ isometric (Prop.~\ref{prop_phipsi},~item~\ref{embeding})}
\end{align*}
from which we deduce using~(\ref{norm_gammaX}) that~$
2g_X q^i= 2g_Y q^i+
\left\| \gamma_{X/Y}^{i}\right\|_X^2$.

Taking again into account orthogonality, we also easily compute the scalar product
\begin{align*}
\scalaire{\gamma_{X/Y}^{i} }{\gamma_{X/Y}^{i+j} }_{X}
&=
\scalaire{\gamma_X^i}{\gamma_X^{i+j}}_X
-
\scalaire{\varphi^*\left(\gamma_Y^i\right)}{\varphi^*\left(\gamma_Y^{i+j}\right)}_X
&& \text{by Def.~\ref{Def_gamma_relative} and orthogonality} \\
&=
\scalaire{\gamma_X^i}{\gamma_X^{i+j}}_X
-
\scalaire{\gamma_Y^i}{\gamma_Y^{i+j}}_Y
&& \text{since $\varphi^*$ isometric} \\
&=
q^i \left((q^j + 1) - \sharp X(\FF_{q^j})\right)
-
q^i \left((q^j + 1) - \sharp Y(\FF_{q^j})\right),
&&\text{by~(\ref{norm_gammaX})}
\end{align*}
as requested.
\end{proof}

We end this Section with a Lemma giving an useful result on the push forward of the relative part of the the $\gamma^i$'s.
\begin{lem}  \label{push_gamma_rel} 
In a tower  $X\rightarrow Y \rightarrow Z$, we have for any $i\geq 0$
$$\phietoilebas{X}{Y}(\gamma^i_{X/Z})=\gamma^i_{Y/Z}.$$
\end{lem}
\begin{proof}
Applying $\phietoilebas{X}{Y}$ to the identity $\gamma^i_X=\varphi^*_{X/Z}(\gamma^i_Z)+\gamma^i_{X/Z}$, we obtain thanks to formula~(\ref{phigamma})
$$\gamma^i_Y=\phietoilebas{X}{Y}\circ\varphi^*_{X/Y}\circ\varphi^*_{Y/Z}(\gamma^i_Z)+\phietoilebas{X}{Y}(\gamma^i_{X/Z}),$$
that is $\gamma^i_Y=\varphi^*_{Y/Z}(\gamma^i_Z) + \phietoilebas{X}{Y}(\gamma^i_{X/Z})$ by Proposition~\ref{prop_phipsi}~item~\ref{circ=Id}, proving the Lemma using Definition~\ref{Def_gamma_relative}.
\end{proof}

%

\section{Applications to relative bounds on numbers of rational points of curves}
We prove Propositions~\ref{X-Y} and~\ref{RelatifNonTrivial} in the first Subsection, so as Theorem~\ref{final} in the second one, in  the very same spirit than in our previous work~\cite[Theorem~11 and Proposition~12, pp. 5420-5421]{WeilNous}.

%

\subsection{First application: number of points in a covering~$X\rightarrow Y$}
As told in the introduction, Propositions~\ref{X-Y} below is well known. We think it is interesting to show how it is neat using the euclidean framework.
\begin{prop} \label{X-Y}
Suppose that there exists a finite morphism $X \rightarrow Y$. Then we have
$$\vert \sharp X(\FF_q)-\sharp Y(\FF_q)\vert \leq 2(g_{X}-g_{Y})\sqrt q.$$
\end{prop}

\begin{proof}
We apply Schwarz inequality
to the relative vectors $\gamma_{X/Y}^{0} $
and $\gamma_{X/Y}^{1} $. We obtain
from Lemma~\ref{normeetpsa2}
\begin{align*}
\left\vert q^0\left(\sharp X(\FF_{q})-\sharp Y(\FF_{q})\right)\right\vert^2
&= \left\vert \scalaire{\gamma_{X/Y}^{0}}{ \gamma_{X/Y}^{1}} \right\vert^2\\
&\leq
\Vert \gamma_{X/Y}^{0}  \Vert_{X}^2 \times \Vert \gamma_X^{1} \Vert_{X}^2\\
&=
2(g_{X}-g_{Y}) q^0\times 2(g_{X}-g_{Y})q^1,
\end{align*}
hence the Proposition.
\end{proof}

The following Proposition~\ref{RelatifNonTrivial} is the relative form of a previous absolute bound~\cite[Proposition~12]{WeilNous}. 
Of course, although less nice, such upper bounds can be given for any size~$\sharp X(\FF_{q^n})$.

\begin{prop} \label{RelatifNonTrivial}
For any finite morphism $X \rightarrow Y$ with $g_{X}\neq g_{Y}$, we have
$$\sharp X({\mathbb F}_{q^2})- \sharp Y({\mathbb F}_{q^2}) \leq 2(g_{X}-g_{Y}) q - \frac{\Bigl(\sharp X({\mathbb F}_{q})-\sharp Y({\mathbb F}_{q})\Bigr)^2}{g_{X}-g_{Y}}.$$
\end{prop}

\begin{proof}
The idea is to write down the matrix~$\Gram(\gamma_{X/Y}^{0} , \gamma_{X/Y}^{1} , \gamma_{X/Y}^{2})$
using Lemma~\ref{normeetpsa2}, and then to use that it 
has a non-negative determinant. In fact, as noted in our previous work \cite{WeilNous}, it is more convenient to write down
\begin{align*}
\Gram\left(q\gamma_{X/Y}^0 + \gamma^2_{X/Y}, \gamma^1_{X/Y}\right)
=
\begin{pmatrix}
4(g_X-g_Y)q^2 + 2q\delta_2
&
2q\delta_1\\
2q\delta_1
&
2(g_X-g_Y)q
\end{pmatrix}
\end{align*}
where we put~$\delta_i = \sharp Y(\FF_{q^i}) - \sharp X(\FF_{q^i})$, $i=1,2$ for short. The result to be proved is just the fact
that this matrix has a non-negative determinant.
\end{proof}

\subsection{Second application: number of points in a commutative diagram}

We focus in this Subection on the situation of a commutative diagram 
\begin{equation} \label{DiagrXY1Y2Z}
\begin{tikzpicture}[>=latex,baseline=(M.center)]
\matrix (M) [matrix of math nodes,row sep=0.65cm,column sep=0.75cm]
{
                              &  |(X)| X & \\
|(Y1)| Y_1  &            & |(Y2)| Y_2 \\
                              &  |(Z)| Z & \\
};
\draw[->] (X) -- (Y1) node[midway,anchor = south east] {$p_1$} ;
\draw[->] (X) -- (Y2) node[midway,anchor = south west] {$p_2$} ;
\draw[->] (Y1) -- (Z) node[midway,below,anchor = north east] {$f_1$} ;
\draw[->] (Y2) -- (Z) node[midway,anchor = north west] {$f_2$} ;
\end{tikzpicture}
\end{equation}
of finite covers
of absolutely irreducible smooth projective curves defined over~$\FF_q$. In order to give a relationship between the number of rational points of the involved curves, we need a decomposition of $\gamma_X^i$, for $i=1, 2$, much sharper than the one given by~(\ref{Decomp_gamma}), taking into account the whole diagram. 


\subsubsection{Pull-back and push-forward morphisms in a commutative diagram}

Applying results of~\S2, we have ten relative linear maps that fit into a diagram of four Euclidean spaces:
\begin{equation}\label{XY1Y2Z}
\begin{tikzpicture}[baseline=(Y1),node distance = 4cm, bend angle=30]
\node (X) {$\Ecal_{X}$};
\node[below left of = X] (Y1) {$\Ecal_{Y_1}$};
\node[below right of = Y1] (Z) {$\Ecal_{Z}$};
\node[below right of = X] (Y2) {$\Ecal_{Y_2}$};
\draw[thick,every node/.style={font=\sffamily\tiny}]
  (Z) edge[right hook-latex] node[midway,above,sloped] {$\varphi^*_{Y_1/Z}$} (Y1)
  (Z.north) edge[right hook-latex, bend right, xshift=3pt] node[midway,above,sloped] {$\varphi^*_{X/Z}$} (X.south)
  (Z) edge[right hook-latex] node[midway,above,sloped] {$\varphi^*_{Y_2/Z}$} (Y2)
  (Y1) edge[right hook-latex] node[midway,below,sloped] {$\varphi^*_{X/Y_1}$} (X)
  (Y2) edge[right hook-latex] node[midway,below,sloped] {$\varphi^*_{X/Y_2}$} (X)
  (X.south) edge[dashed,bend right,->,>=latex,xshift=-3pt] node[midway,below,sloped] {$\phietoilebas{X}{Z}$} (Z.north) 
  (Y1.south) edge[dashed,bend right,->,>=latex] node[midway,below,sloped] {$\phietoilebas{Y_1}{Z}$} (Z.west) 
  (Y2.south) edge[dashed,bend left,->,>=latex] node[midway,below,sloped]  {$\phietoilebas{Y_2}{Z}$} (Z.east)
  (X.west) edge[dashed,bend right,->,>=latex] node[midway,above,sloped] {$\phietoilebas{X}{Y_1}$} (Y1.north)
  (X.east) edge[dashed,bend left,->,>=latex] node[midway,above,sloped]  {$\phietoilebas{X}{Y_2}$} (Y2.north) ;
\end{tikzpicture}
\end{equation}
As noted in the proof of Proposition~\ref{prop_phipsi}, all the involved square morphisms~$f_i\times f_i$ and~$p_i\times p_i$ from a square surface to
another are proper and flat. As a consequence, the push-forward and pull-back operations are functorial
\cite[\S1.4, p~11 \& \S1.7, p~18]{Fulton}, that is we have~$\phietoilebas{X}{Z} = \phietoilebas{Y_i}{Z}\circ\phietoilebas{X}{Y_i}$
and~$\varphi_{X/Z}^* = \varphi_{X/Y_i}^* \circ \varphi_{Y_i/Z}^*$ for~$i=1,2$. We also recall that all the $\varphi_{\_/\_}^*$ maps are isometric embeddings by Proposition~\ref{prop_phipsi}.

\bigskip

In order to understand better the relationships between these euclidean vector spaces and linear maps, we need a new hypothesis in the following Lemma. 

\begin{lem} \label{prop_XY12Z}
Let a commutative diagram of curves like in~\eqref{DiagrXY1Y2Z}. Suppose that the fiber product~$Y_1\times_Z Y_2$ is absolutely irreducible and
smooth. Then, we have: 
\begin{enumerate}
\item \label{phi_psi_XY12Z} 
$\phietoilebas{X}{Y_2} \circ \varphi^*_{X/Y_1}
=
\varphi^*_{Y_2/Z} \circ \phietoilebas{Y_1}{Z}$
on~$\Ecal_{Y_1}$;
\item \label{item_remontes_perp}
inside~$\Ecal_X$, the subspaces~$\varphi^*_{X/Y_1}(\Ecal_{Y_1/Z})$
and~$\varphi^*_{X/Y_2}(\Ecal_{Y_2/Z})$ are orthogonal, and lie into $\varphi^*_{X/Z}(\Ecal_{Z})^{\perp}=\Ecal_{X/Z}$.
\end{enumerate}
\end{lem}

\begin{proof}
Let us prove the first item. By the universal property of the fiber product,
the two top morphisms in the commutative starting diagram~(\ref{DiagrXY1Y2Z}) factor through
$$
\begin{matrix}
g : & X &\longrightarrow &Y_1\times_Z Y_2 \subset Y_1\times Y_2\\
~& x & \mapsto & (p_1(x), p_2(x)),
\end{matrix}
$$
yielding to the diagram
\begin{align*}
\begin{tikzpicture}[>=latex,baseline=(Y1),node distance = 1.5cm, bend angle=30]
\node (Y12) {$Y_1 \times_Z Y_2$};
\node[above of = Y12] (X) {$X$};
\node[below left of = Y12] (Y1) {$Y_1$};
\node[below right of = Y1] (Z) {$Z$};
\node[below right of = Y12] (Y2) {$Y_2$};
\draw[thick,every node/.style={font=\sffamily\small}]
  (Y1) edge[->,right] node[midway,below,sloped] {$f_1$} (Z)
  (Y2) edge[->,right] node[midway,below,sloped] {$f_2$} (Z)
  (Y12) edge[->,right] node[midway,below,sloped] {$\pi_1$} (Y1)
  (Y12) edge[->,right] node[midway,below,sloped] {$\pi_2$} (Y2)
  (X) edge[->,right] node[midway,left] {$g$} (Y12)
  (X.west) edge[bend right,->] node[midway,above,sloped] {$p_1$} (Y1.north)
  (X.east) edge[bend left,->] node[midway,above,sloped] {$p_2$} (Y2.north) ;
\end{tikzpicture},
%
\end{align*}
where $\pi_i$ is the projection of the fiber product~$Y_1\times_ZY_2 \subset Y_1\times Y_2$ on the $i$-th factor~$Y_i$ and $p_i=\pi_i\circ g$.
This last diagram induces a similar one between the five squared curves, related by the seven squared morphisms.
As already noted, all these squared morphisms are proper and flat, the flatness of~$g\times g$ following from the smoothness assumption ont~$Y_1 \times_Z Y_2$.

Since the bottom square involving the nodes~$(Y_1 \times_Z Y_2)^2$, $Y_i^2$ $i=1,2$, and~$Z^2$ is itself a fiber square,
we know that~\cite[Prop~1.7 p.18]{Fulton},
\begin{equation} \label{2020}
(\pi_{2}\times \pi_{2})_*\circ (\pi_1\times \pi_1)^* = (f_2\times f_2)^*\circ(f_1\times f_1)_*
\end{equation}
on~$\Num(Y_1\times Y_1)_\RR$. Since~$g\times g$ is also finite and flat, we also know
that~$(g\times g)_*\circ (g\times g)^* = (\deg g)^2 \Id_{\Num(Y_1\times Y_1)_\RR}$ \cite[Ex~1.7.4 p.~20]{Fulton}.
 Taking into account normalizations~(\ref{phipsi}),  item~\ref{phi_psi_XY12Z} follows now by direct calculation from~
(\ref{2020}) and the multiplicativity of the degree in towers of finite morphisms.


For the second item, we first prove that $\varphi^*_{X/Y_i}(\Ecal_{Y_i/Z}) \subset \Ecal_{X/Z}=\left(\varphi^*_{X/Z}(\Ecal_Z)\right)^\perp$. Let $\gamma_Z \in \Ecal_Z$ and $\gamma_i \in \Ecal_{Y_i/Z}$ for $i=1$ or $2$. We have
\begin{align*}
\scalaire{\varphi^*_{X/Z}(\gamma_Z)}{\varphi^*_{X/Y_i}(\gamma_i)}_X
&=
\scalaire{\varphi^*_{X/Y_i} \circ \varphi^*_{Y_i/Z}(\gamma_Z)}{\varphi^*_{X/Y_i}(\gamma_i)}_X
\\
&= \scalaire
{\varphi^*_{Y_i/Z}(\gamma_Z)}
{\gamma_i}_{Y_i}
&&\text{since $\varphi^*_{X/Y_i}$ is an isometry}
\\
&=0 &&\begin{array}{l} \text{since~} \varphi^*_{Y_i/Z}(\gamma_Z)\in\varphi^*_{Y_i/Z}(\Ecal_Z)\\\text{and~} \gamma_i\in \Ecal_{Y_i/Z}=\left(\varphi^*_{Y_i/Z}(\Ecal_Z)\right)^{\perp}\end{array}.
\end{align*}
Last, let  $\gamma_1 \in \Ecal_{Y_1/Z}$. Then, we have by item~\ref{phi_psi_XY12Z} together with Lemma~\ref{phi_*zero}
$$\phietoilebas{X}{Y_2} \circ \varphi^*_{X/Y_1}(\gamma_1) =\varphi^*_{Y_2/Z} \circ \phietoilebas{Y_1}{Z} (\gamma_1) =0.$$
It follows by adjunction that, for any $\gamma_2 \in \Ecal_{Y_2}$, we have
\begin{align*}
\scalaire{\varphi^*_{X/Y_1}(\gamma_1)}{\varphi^*_{X/Y_2}(\gamma_2)}_X
&=\scalaire{\phietoilebas{X}{Y_2} \circ \varphi^*_{X/Y_1}(\gamma_1)}{\gamma_2}_{Y_2}\\
&= \scalaire{0}{\gamma_2}_{Y_2}\\
&=0,
\end{align*}
so that $\varphi^*_{X/Y_1}(\Ecal_{Y_1/Z}) \subset \left(\varphi^*_{X/Y_2}(\Ecal_{Y_2})\right)^{\perp} \subset \left(\varphi^*_{X/Y2}(\Ecal_{Y_2/Z})\right)^{\perp}$, and the proof is complete.
\end{proof}

\subsubsection{The relative part of the~$\gamma_X^i$'s in a commutative diagram} 

We are now ready to introduce some  orthogonal decomposition of the~$\gamma_X^i$'s inside~$\Ecal_X$ sharper than the one

\begin{equation} \label{Decomp_gamma_X/Z}
\gamma^i_X = \underbrace{\varphi_{X/Z}^*(\gamma_Z^i)}_{\in \varphi_{X/Z}^*(\Ecal_Z)} + \underbrace{\gamma^i_{X/Z}}_{\in \Ecal_{X/Z}}
\end{equation}
given in Section~2 for the covering $X \rightarrow Z$, that takes into account
the whole diagram~(\ref{DiagrXY1Y2Z}) below~$X$.
%
%
%

There is, from item~\ref{item_remontes_perp} of Lemma~\ref{prop_XY12Z}, an orthogonal decomposition of $\Ecal_{X/Z}$ of the form
%
%
%
%
\begin{align} \label{Decomp_Fine_E_X}
\Ecal_{X/Z}
&=
\varphi^*_{X/Y_1}(\Ecal_{Y_1/Z})
\oplus
\varphi^*_{X/Y_2}(\Ecal_{Y_2/Z})
\oplus
\Ecal_{12}
\end{align}
for some uniquely defined subspace $\Ecal_{X/Y_1, Y_2/Z}=\Ecal_{12}$ for simplicity. To study the corresponding decomposition of the relative
vectors~$\gamma^i_{X/Z} \in \Ecal_{X/Z}$ for $X\rightarrow Z$, we need another definition.

\begin{defi}\label{def_gamma_12}
For~$i\geq 0$, we put
\begin{equation}
\gamma_{12}^i = \gamma_{X/Z}^i - \varphi^*_{X/Y_1}\left(\gamma^i_{Y_1/Z}\right) - \varphi^*_{X/Y_2}\left(\gamma^i_{Y_2/Z}\right),
\end{equation}
and we call it the $i$-th ``square diagram'' part of the Frobenius.
\end{defi}

\begin{lem} \label{lem_Decomp_Fine_gamma_X/Z}
Consider the situation of diagram~(\ref{DiagrXY1Y2Z}) in which $Y_1\times_Z Y_2$ is assumed to be absolutely irreduible and smooth. Let $i\geq 0$. Then the decomposition of $\gamma^i_{X/Z}$ as an orthogonal sum accordingly to~(\ref{Decomp_Fine_E_X}) is given by
\begin{equation} \label{Decomp_Fine_gamma_X/Z}
\gamma^i_{X/Z}
=
\underbrace{\varphi^*_{X/Y_1}\left(\gamma^i_{Y_1/Z}\right)}_{\in \varphi^*_{X/Y_1}\left(\Ecal_{Y_1/Z}\right)}
+
\underbrace{\varphi^*_{X/Y_2}\left(\gamma^i_{Y_2/Z}\right)}_{\in \varphi^*_{X/Y_2}\left(\Ecal_{Y_2/Z}\right)}
+
\underbrace{\gamma^i_{12}}_{\in \Ecal_{12}}.
\end{equation}
\end{lem}

\begin{proof}
Given Definition~\ref{def_gamma_12}, formula~(\ref{Decomp_Fine_gamma_X/Z}) clearly holds true. Since the vectors $\varphi^*_{X/Y_1}\left(\gamma^i_{Y_1/Z}\right)$ and $\varphi^*_{X/Y_2}\left(\gamma^i_{Y_2/Z}\right)$ are orthogonal thanks to Lemma~\ref{prop_XY12Z}, item~\ref{item_remontes_perp}, it suffices to prove that $\gamma_{12}^i \perp \varphi^*_{X/Y_k}\left(\gamma^i_{Y_k/Z}\right)$ for $k=1, 2$. Let for instance $k=1$. Then, we have
\begin{align*}
\scalaire{\gamma^i_{12}}{\varphi^*_{X/Y_1}(\gamma^i_{Y_1/Z})}_X
&=
\scalaire{\gamma^i_{X/Z}}{\varphi^*_{X/Y_1}(\gamma^i_{Y_1/Z})}_X\\
&\quad- \scalaire{\varphi^*_{X/Y_1}(\gamma^i_{Y_1/Z})}{\varphi^*_{X/Y_1}(\gamma^i_{Y_1/Z})}_X \\
&\quad-\scalaire{\varphi^*_{X/Y_2}(\gamma^i_{Y_2/Z})}{\varphi^*_{X/Y_1}(\gamma^i_{Y_1/Z})}_X \\
&=
\scalaire{\phietoilebas{X}{Y_1}(\gamma^i_{X/Z})}{\gamma^i_{Y_1/Z}}_X &&\text{by adjunction}\\
&\quad- \scalaire{\gamma^i_{Y_1/Z}}{\gamma^i_{Y_1/Z}}_X && \text{since $\varphi^*_{X/Y_1}$ is isometric}\\
&\quad-0 &&\text{by Lemma~\ref{prop_XY12Z}, item~\ref{item_remontes_perp}}\\
&=
\scalaire{\gamma^i_{Y_1/Z}}{\gamma^i_{Y_1/Z}}_X &&\text{by Lemma~\ref{push_gamma_rel}}\\
&\quad- \scalaire{\gamma^i_{Y_1/Z}}{\gamma^i_{Y_1/Z}}_X && \\
&=0,
\end{align*}
and the proof is complete.
\end{proof}

Next, we can compute the norms and scalar products of the~$\gamma^i_{12}$'s.

\begin{lem} \label{normeetpsa}
Let a commutative diagram of curves like in~\eqref{DiagrXY1Y2Z}.
Suppose that $Y_1\times_Z Y_2$ is absolutely irreducible and smooth.
Then for any $i \geq 0$, $j >0$, we have
$$\Vert \gammaa^i\Vert_X = \sqrt{2(g_{X}-g_{Y_1}-g_{Y_2}+g_Z)q^i}$$
and
$$\scalaire{\gammaa^i}{\gammaa^{i+j} }_{X}
=
q^i\left(\sharp Y_1(\FF_{q^j})+\sharp Y_2(\FF_{q^j})- \sharp X(\FF_{q^j})- \sharp Z(\FF_{q^j})\right).$$
\end{lem}

\begin{proof}
From the orthogonal sum
$$
\gamma_{X/Z}^{i}
=
\varphi_{X/Y_1}^*\left(\gamma_{Y_1/Z}^{i}\right)
+
\varphi_{X/Y_2}^*\left(\gamma_{Y_2/Z}^{i}\right)
+
\gammaa^{i},
$$
we get using pythagore
\begin{align*}
\left\|\gamma_{X/Z}^{i}\right\|_X^2
=
\left\|\varphi_{X/Y_1}^*\left(\gamma_{Y_1/Z}^{i}\right)\right|_X^2
+
\left\|\varphi_{X/Y_2}^*\left(\gamma_{Y_2/Z}^{i}\right)\right\|_X^2
+
\left\|\gammaa^{i}\right\|_X^2,
\end{align*}
and also
\begin{align*}
\scalaire{\gamma_{X/Z}^{i}}{\gamma_{X/Z}^{i+j}}_X
&=
\scalaire{\varphi_{X/Y_1}^*\left(\gamma_{Y_1/Z}^{i}\right)}{\varphi_{X/Y_1}^*\left(\gamma_{Y_1/Z}^{i+j}\right)}_X
+
\scalaire{\varphi_{X/Y_2}^*\left(\gamma_{Y_2/Z}^{i}\right)}{\varphi_{X/Y_2}^*\left(\gamma_{Y_2/Z}^{i+j}\right)}_X\\
&\quad +
\scalaire{\gammaa^{i}}{\gammaa^{i+j}}_X.
\end{align*}
This permits to conclude using lemma~\ref{normeetpsa2} and the fact that all the maps~$\varphi^*_{\_/\_}$ are isometries.
\end{proof}

\subsubsection{Number of rational points in a commutative diagram} 

We can now prove the following result.

\begin{theo} \label{final}
Let $X, Y_1, Y_2$ and $Z$ be absolutely irreducible smooth projective curves in a commutative diagram~\eqref{DiagrXY1Y2Z} of finite morphisms. Suppose that the fiber
product~$Y_1\times_Z Y_2$ is absolutely irreducible and smooth. Then
$$
\left\vert \sharp X(\FF_q)-\sharp Y_1(\FF_q)-\sharp Y_2(\FF_q)+\sharp Z(\FF_q)\right\vert
\leq 2(g_{X}-g_{Y_1}-g_{Y_2}+g_Z) \sqrt q.
$$
\end{theo}

\begin{proof} In the same way than for the proof of Proposition~\ref{X-Y}, this is Schwarz inequality for $\gammaa^0$
and~$\gammaa^1$ together with Lemma~\ref{normeetpsa}.
\end{proof}

\medskip

It worth to notice that Theorem~\ref{final} cannot holds without any hypothesis. For instance, if $X=Y_1=Y_2$ and the morphisms $Y_i \rightarrow Z$ are the same, then the right hand side equals $2(g_X-2g_X+g_Z)\sqrt{q}=-2(g_X-g_Z)\sqrt{q}$, a negative number! In this case, the Theorem doesn't apply since  $Y_1\times_Z Y_2$ is not irreducible.

\begin{rk}
For~$Y_1 \times_Z Y_2$ to be absolutely irreducible, it suffices that
the tensor product of function fields~$\FF_q(Y_1) \otimes_{\FF_q(Z)} \FF_q(Y_2)$
to be an integral domain.
For~$Y_1 \times_Z Y_2$ to be smooth at a point~$(Q_1,Q_2)$,
it is necessary and sufficient that at least one of the
morphisms~$Y_i \to Z$ is unramified at~$Q_i$.
\end{rk}

\bibliographystyle{amsalpha}
\bibliography{Weilrelatif}

\providecommand{\bysame}{\leavevmode\hbox to3em{\hrulefill}\thinspace}
\providecommand{\MR}{\relax\ifhmode\unskip\space\fi MR }
\providecommand{\MRhref}[2]{%
  \href{http://www.ams.org/mathscinet-getitem?mr=#1}{#2}
}
\providecommand{\href}[2]{#2}
\begin{thebibliography}{Wei48}

\bibitem[AP95]{AubryPerret}
Yves Aubry and Marc Perret, \emph{Coverings of singular curves over finite
  fields}, Manuscripta Math. \textbf{88} (1995), no.~4, 467--478. \MR{1362932
  (97g:14022)}

\bibitem[Ful98]{Fulton}
William Fulton, \emph{Intersection theory}, Springer, 1998.

\bibitem[Har77]{Hartshorne}
Robin Hartshorne, \emph{Algebraic geometry}, Graduate Texts in Mathematics,
  vol.~52, Springer, 1977.

\bibitem[HP19]{WeilNous}
Emmanuel Hallouin and Marc Perret, \emph{An unified viewpoint for upper bounds
  for the number of points of curves over finite fields via euclidean geometry
  and semi-definite symmetric toeplitz matrices}, Trans. Amer. math. Soc.
  \textbf{312} (2019), 5409--5451.

\bibitem[Tsf92]{T92}
Michael~A. Tsfasman, \emph{Some remarks on the asymptotic number of points},
  Coding theory and algebraic geometry ({L}uminy, 1991), Lecture Notes in
  Math., vol. 1518, Springer, Berlin, 1992, pp.~178--192. \MR{1186424
  (93h:11064)}

\bibitem[VD83]{VD}
Sergei~G. Vl{\u{a}}du{\c{t}} and Vladimir Drinfeld, \emph{Number of points of
  an algebraic curve}, Funksional Anal i Prilozhen \textbf{17} (1983), 53--54.

\bibitem[Wei48]{WeilCA}
Andr{\'e} Weil, \emph{Courbes alg\'ebriques et vari\'et\'es ab\'eliennes},
  Hermann et Cie., Paris, 1948.

\end{thebibliography}
\end{document}